\documentclass[a4paper, 11pt]{article}
\usepackage[mathscr]{eucal}
\usepackage{amssymb}
\usepackage{latexsym}
\usepackage{amsthm}
\usepackage{mathrsfs}
\usepackage{amsmath}
\usepackage[dvips]{graphicx}
\usepackage{psfrag}
\usepackage{a4wide}
\usepackage{lscape}

\newcommand{\mb}{\mathbf}
\newcommand{\R}{\mathbb{R}}

\newtheorem{theorem}{Theorem}[section]
\newtheorem{proposition}[theorem]{Proposition}
\newtheorem{lemma}[theorem]{Lemma}

\theoremstyle{definition}

\newtheorem{remark}[theorem]{Remark}

\makeatletter
\@addtoreset{equation}{section}

\makeatother

\title{Maximal $m$-distance sets containing the representation of the Hamming graph $H(n,m)$} 

\author{
Saori Adachi, Rina Hayashi, Hiroshi Nozaki, and Chika Yamamoto
}
\begin{document}
\maketitle

\renewcommand{\thefootnote}{\fnsymbol{footnote}}
\footnote[0]{2010 Mathematics Subject Classification: 
05D05  	
(05B30)  	
\\
\noindent
{\it Saori Adachi}: 
	Chiryu-higashi High School,  
	18-6 Oyama, Nagashino-cho, 
	Chiryu, Aichi 472-8639, 
	Japan.
	k630945d@m2.aichi-c.ed.jp.
\\ \noindent
{\it Hiroshi Nozaki}: 
	Department of Mathematics Education, 
	Aichi University of Education, 
	1 Hirosawa, Igaya-cho, 
	Kariya, Aichi 448-8542, 
	Japan.
	hnozaki@auecc.aichi-edu.ac.jp.
\\ \noindent
{\it Rina Hayashi, Chika Yamamoto}: 
They were under-graduate students in 
Aichi University of Education in the period 2010--2014. 
}

\begin{abstract}
A set $X$ in the Euclidean space $\mathbb{R}^d$ is an {\it $m$-distance set} if the set of Euclidean distances between two distinct points in $X$ has size $m$. An $m$-distance set $X$ in $\mathbb{R}^d$ is {\it maximal} if there does not exist a vector $\mb{x}$ in $\mathbb{R}^d$ such that the union of $X$ and $\{\mb{x}\}$ still has only $m$ distances.  
Bannai, Sato, and Shigezumi (2012) investigated  maximal $m$-distance sets that contain the Euclidean representation of the Johnson graph $J(n,m)$.  
In this paper, we consider the same problem for the Hamming graph $H(n,m)$. The Euclidean representation of $H(n,m)$ is an $m$-distance set in $\mathbb{R}^{m(n-1)}$.
We prove that 
if the representation of $H(n,m)$
is not maximal as an $m$-distance set 
for some $m$, 
then the maximum value of $n$ is $m^2 + m - 1$. 
Moreover we classify the largest 
$m$-distance sets that contain the representation of $H(n,m)$ for $n\geq 2$ and  $m\leq 4$. 
We also classify the maximal $2$-distance sets that are in $\mathbb{R}^{2n-1}$ and contain the representation of $H(n,2)$ for $n\geq 2$. 
\end{abstract}
\textbf{Key words}:
Hamming graph, few-distance set,  Euclidean representation, Erd\H{o}s--Ko--Rado theorem.  

\section{Introduction}
A subset $X$ of the Euclidean space $\R^d$ is 
 an {\it $m$-distance set} 
if the size of the set of distances between two distinct points in $X$ is equal to $m$. The size of an $m$-distance set is bounded above by $\binom{d+m}{m}$ \cite{BBS}.
One of major problems is to find the maximum
possible cardinality of an $m$-distance set for given $m$ and $d$.  
The largest $1$-distance set in $\R^d$ 
is the regular simplex for $d\geq 1$, and it has $d+1$ points.  
Largest $2$-distance sets in $\R^d$ are classified for $d\leq 7$ \cite{ES66,L97}. 
Lison\v{e}k \cite{L97} constructed a largest $2$-distance set in $\R^8$, which is the only known set attaining the bound $|X|\leq \binom{d+m}{m}$ for $m\geq 2$.  
  Largest $m$-distance sets in $\R^2$ are classified for $m\leq 5$ \cite{EF96,S04,S08}. 
Two largest $6$-distance sets are known \cite{W12}. 
Tables \ref{tb:1}, \ref{tb:2}  show the cardinalities $|X|$ of largest distance sets $X$, and the number $\#$ of the sets, up to isometry.
\begin{table}
\begin{center}
\caption{$m=2$}\label{tb:1}
\begin{tabular}{c|ccccccc}
$d$ &  2 & 3 & 4 & 5 & 6 & 7 & 8\\
\hline
$|X|$ & 5 & 6 & 10 & 16 & 27 & 29 & 45\\
\hline 
  $\#$  & 1 & 6 & 1  & 1  & 1  & 1  & $\geq 1$ 
\end{tabular}
\caption{$d=2$}\label{tb:2}
\begin{tabular}{c|ccccc  } 
$m$ &  2 & 3 & 4 & 5 & 6\\
\hline
$|X|$& 5 & 7 & 9 & 12 &13 \\
\hline 
 $\#$   & 1 & 2 & 4 & 1& $\geq$ 2\\ 
\end{tabular}
\end{center}
\end{table}
The largest $3$-distance set in $\R^3$ is the vertex set of the icosahedron \cite{Spre}.

The Euclidean representation $\tilde{J}(n,m)$ of the Johnson scheme $J(n,m)$ is the subset of 
$\R^{n}$ consisting of all vectors 
with $1$'s in $m$ coordinates and $0$'s elsewhere. 
The set $\tilde{J}(n,m)$ with $n\geq 2m$ can be interpreted as an $m$-distance set in $\R^{n-1}$ because the sum of entries of each element is $m$.   
The largest known $m$-distance sets in $\R^{n-1}$ are mostly $\tilde{J}(n,m)$. 
An $m$-distance set $X$ in $\R^n$ is {\it maximal} if there does not exist 
$\mb{x}\in \R^n$ such that $X\cup\{\mb{x}\}$ is still $m$-distance. 
Bannai, Sato, and Shigezumi \cite{BSS12}
investigated maximal $m$-distance sets 
that are in $\R^{n-1}$ and contain $\tilde{J}(n,m)$. 
They gave a necessary and sufficient condition for $\tilde{J}(n,m)$ to be a maximal $m$-distance set in $\R^{n-1}$, and  
 classified the largest $m$-distance sets
containing $\tilde{J}(n,m)$ for $n\geq 2$ and $m\leq 5$, except for $(n,m)=(9,4)$. 
The case $(n,m)=(9,4)$ is solved in
\cite{ANX}. 
This construction of distance sets might be possible for 
other association schemes. 
In this paper we consider the Hamming scheme $H(n,m)$.

Let $F_n=\{1,\ldots,n\}$, $\mathbf{x}=(x_1,\ldots,x_m)\in F_n^m$, and
$\mathbf{y}=(y_1,\ldots,y_m)
\in F_n^m$. 
The Hamming distance of $\mathbf{x}$ and $\mathbf{y}$ is defined to be $h(\mathbf{x},\mathbf{y})=|\{i \colon\, x_i \ne y_i \}|$.
The Hamming scheme $H(n,m)$ is an association scheme $(F_n^m, \{R_0,\ldots,R_m\})$, where 
$
R_i=\{(\mathbf{x},\mathbf{y})\colon\, h(\mathbf{x},\mathbf{y})=i  \}. 
$
Let $\varphi: F_n^m \rightarrow \R^{mn}$ be the embedding defined by
\[
\varphi: \mb{x}=(x_1, \ldots, x_m) \mapsto 
\tilde{\mb{x}}=\sum_{i=1}^m \mb{e}_{(i-1)n+x_i},
\]
where $\{\mb{e}_1,\ldots,\mb{e}_{mn}\}$ is the standard basis of $\R^{mn}$. 
Let $\tilde{H}(n,m)$ denote the image of $\varphi$. 
Note that $h(\mathbf{x},\mb{y})=k$ for $\mb{x},\mb{y} \in H(n,m)$ if and only if $d(\tilde{\mb{x}},\tilde{\mb{y}})=\sqrt{2k}$ for $\tilde{\mb{x}},\tilde{\mb{y}}\in \tilde{H}(n,m)$, where $d(,)$ is the Euclidean distance. 
Let $\mathbf{j}_k$ denote the vector 
\[
\mathbf{j}_k=\sum_{i=(k-1)n+1}^{kn} \mb{e}_i.
\]
Every vector in $\tilde{H}(n,m)$ is perpendicular to $\mathbf{j}_k$ for $k\in \{1,\ldots, m\}$. We can therefore interpret
$\tilde{H}(n,m)$ as an $m$-distance set in $\R^{m(n-1)}$. 
We consider maximal $m$-distance sets that are in $\R^{m(n-1)}$ and contain $\tilde{H}(n,m)$. 

This paper is summarized as follows.    
In Section \ref{sec:2}, 
we give some notation, and determine the 
coordinates of a vector $\mb{x}$ when $\mb{x}$ can be added to $\tilde{H}(n,m)$ while maintaining $m$-distance.   
In Section \ref{sec:3}, the maximal $2$-distance sets containing $\tilde{H}(n,2)$ are classified by an explicit way. 
In Section \ref{sec:4}, 
we give a necessary and sufficient condition for 
 $\tilde{H}(n,m)$ to be maximal as an $m$-distance set. 
Moreover, we prove that 
if $\tilde{H}(n,m)$ is not maximal as an $m$-distance set for some $m$, then
the maximum value of $n$ is equal to $m^2+m-1$. 
In Section \ref{sec:5}, 
we classify the largest $m$-distance sets that are in $\R^{m(n-1)}$ and contain $\tilde{H}(n,m)$ for $n\geq 2$ and  $m\leq 4$. 
Tables \ref{tb:3}--\ref{tb:5} show the maximum cardinalities $|X|$ and dimension
$d=m(n-1)$.    
In Section \ref{sec:6}, 
we classify maximal $2$-distance sets
that are
in $\mathbb{R}^{2(n-1)+1}$ and contain $\tilde{H}(n,2)$. 

\begin{table}\begin{center}
\caption{$m=2$}\label{tb:3}
  \begin{tabular}{c|c} 
    $n$ & $5$ \\ \hline
    $d$ & $8$ \\ \hline
    $|X|$ & $40$
  \end{tabular}
\caption{$m=3$}\label{tb:4}
  \begin{tabular}{c|cccc} 
    $n$ & $3$ & $5$ & $9$ & $11$ \\ \hline
    $d$ & $6$ & $12$ & $24$ & $30$ \\ \hline
    $|X|$ & $40$ & $200$ & $981$ & $1451$ 
  \end{tabular}
\caption{$m=4$} \label{tb:5}
 \begin{tabular}{c|cccccccccc} 
    $n$ & $2$ & $3$ & $5$ & $6$ &$7$ & $9$ & $11$ & $13$ & $14$ & $19$ \\ \hline
    $d$ & $4$ & $8$ & $16$ & $20$ &$24$  & $32$ & $40$ & $48$ & $52$ & $72$ \\ \hline
    $|X|$ & $25$ & $222$ & $1600$ & $2004$ & $3390$ & $8829$ & $16566$ & $29056$ & $39417$ & $133381$ 
     \end{tabular} \end{center}
\end{table}
\section{Vectors that can be added to $\tilde{H}(n,m)$}  \label{sec:2}
First we give some notation. 
For real numbers $x_1,\ldots, x_n$ and natural numbers $\lambda_1,\ldots, \lambda_n$, we use the notation  \[(x_1^{\lambda_1},\ldots, x_n^{\lambda_n} )=
(\underbrace{x_1,\ldots,x_1}_{\lambda_1}, \ldots, \underbrace{x_n,\ldots,x_n}_{\lambda_n})\in \R^{ \lambda_1+\cdots+\lambda_n}, \]
and $(x_1^{\lambda_1})$ is abbreviated to $x_1^{\lambda_1}$. 
Let $\mb{x}=(x_1,\ldots,x_n)\in \R^n$ 
and  $X_i \subset \R^{n}$.  
We use the following notation:  
\[
\mb{x}^P=(x_1,\ldots,x_n)^P=
\{(x_{\sigma(1)},\ldots, x_{\sigma(n)}) \colon\, \sigma \in S_n\} \subset \R^n,
\]
\[
(X_1,\ldots, X_m)=
\{(\mb{x}_{1},\ldots,\mb{x}_{m})
\colon\, \mb{x}_i \in X_i\}
\subset \R^{mn}, 
\]
\[
(X_1,\ldots, X_m)^P=
\{(\mb{x}_{\sigma(1)},\ldots,\mb{x}_{\sigma(m)})
\colon\, \mb{x}_i \in X_i, \sigma \in S_m\}
\subset \R^{mn}, 
\]
where $S_n$ is the permutation group. A one-element set $\{\mb{x}\}$ is abbreviated to $\mb{x}$ in these expressions.

Suppose $\mb{x} \in \R^{mn}$ can be added to $\tilde{H}(n,m)$ while maintaining $m$-distance.  
 Note that $d(\mb{x},\mb{y})\in \{\sqrt{2}, \sqrt{4},\ldots, \sqrt{2m}\}$ for each $\mb{y} \in \tilde{H}(n,m)$. 
For $\mb{x}=(\mb{x}_1,\ldots, \mb{x}_m)$ with 
$\mb{x}_1,\ldots, \mb{x}_m \in \R^n$, 
the sum of all entries of $\mb{x}_i$ is $1$ for each $i \in \{1,\ldots,m\}$.  Because of the automorphism group of $\tilde{H}(n,m)$,  each vector in $(\mb{x}_1^P,\ldots, \mb{x}_m^P)^P$ can be added to $\tilde{H}(n,m)$.

For $\mb{x}=(x_1,\ldots, x_{mn})$, and $\mb{y},\mb{y}' \in \tilde{H}(n,m)$
such that $\mb{y}=\sum_{i=1}^m \mb{e}_{(i-1)n+q_i}$ and $\mb{y}'=\sum_{i=1}^m \mb{e}_{(i-1)n+q_i'}$, 
we obtain
\begin{equation} \label{eq:*}
\sum_{i=1}^m(x_{(i-1)n+q_i}-x_{(i-1)n+q_i'})\in \{0,\pm 1,\ldots, \pm (m-1)\}
\end{equation}
from $d(\mb{x},\mb{y})^2-d(\mb{x},\mb{y}')^2$. 
Let $j\in \{1,\ldots, m\}$ be fixed.
If $q_j\ne q_j'$ and $q_i=q_i'$ for $i\ne j$, then 
\[
x_{(j-1)n+q_j}-x_{(j-1)n+q_j'}\in \{0,\pm 1,\ldots, \pm (m-1)\}.
\]
 Thus the block $\mb{x}_j$  satisfies
\[
\mb{x}_j \in \big(\alpha_j^{k_1^{(j)}},
(\alpha_j-1)^{k_2^{(j)}},\ldots,(\alpha_j-t_j+1)^{k_{t_j}^{(j)}}\big)^P
\]
for some $\alpha_j \in \R$, $t_j \in \mathbb{Z}$ with $1 \leq t_j \leq m$, and $k_i^{(j)} \in \mathbb{Z}$ such that $0\leq k_i^{(j)}$ and $\sum_{i=1}^{t_j}k_i^{(j)}
= n$. By \eqref{eq:*} and $1 \leq t_j \leq m$, we have  
\begin{equation}\label{eq:t_j}
 \sum_{j=1}^m t_j \leq 2m-1. 
\end{equation}
Since the sum of all entries of 
$\mb{x}_j$ is $1$, it follows that
\[
n \alpha_j= 1+ \sum_{i=1}^{t_j} (i-1) k_i^{(j)}.
\]
Let $k_0^{(j)}=1+ \sum_{i=1}^{t_j} (i-1) k_i^{(j)}\in \mathbb{Z}$.  
Now we have
\[
\mb{x}_j \in \left(\frac{k_0^{(j)}}{n}^{k_1^{(j)}},
\left(\frac{k_0^{(j)}}{n}-1\right)^{k_2^{(j)}},\ldots,\left(\frac{k_0^{(j)}}{n}-t_j+1\right)^{k_{t_j}^{(j)}}\right)^P.
\] 

For $\mb{x}$ with $\mb{x}_j=(x_{1},\ldots,x_{n})$ and $\mb{y}$ with $\mb{y}_j= \mb{e}_{q_j}$, 
let 
\[
l_i^{(j)}=\begin{cases}
1 \text{ if $x_{q_j}=k_0^{(j)}/n-i+1$},\\
0 \text{ otherwise}.
\end{cases}
\]
 The squared distance $d(\mb{x}, \mb{y})^2$
can be expressed by  
\begin{align} \nonumber
d(\mb{x}, \mb{y})^2
&=\sum_{j=1}^m \left(
\sum_{i=1}^{t_j} l_i^{(j)}\left(1-\left(\frac{k_0^{(j)}}{n}-i+1 \right)\right)^2+
\sum_{i=1}^{t_j} (k_i^{(j)}-l_i^{(j)})\left(\frac{k_0^{(j)}}{n}-i+1 \right)^2 \right)\\
\label{eq:d(x,y)}&=
\sum_{j=1}^m\left(1-n-2k_0^{(j)}-\frac{{k_0^{(j)}}^2}{n}
+ \sum_{i=1}^{t_j}i^2 k_i^{(j)}
+ 2\sum_{i=1}^{t_j} il_i^{(j)}
\right).
\end{align}

 \section{Case $m=2$}\label{sec:3}
In this section, we demonstrate a manner of classifying maximal 
$2$-distance sets which contain $\tilde{H}(n,2)$ as easy case. 
Suppose $\mb{x}\in \R^{2n}$ can be added to $\tilde{H}(n,2)$ while maintaining $2$-distance. 
By \eqref{eq:t_j}, we have $(t_1,t_2)=(1,1), (2,1)$, or $(1,2)$. 
Thus $\mb{x}$ can be expressed by 
\[
\mb{x} \in  \left(\left(\frac{k_0^{(1)}}{n}^{k_1^{(1)}},
\left(\frac{k_0^{(1)}}{n}-1\right)^{n-k_1^{(1)}}\right)^P,
\frac{1}{n}^n
\right)^P
\]
such that $k_0^{(1)}=1+n-k_1^{(1)}$. 
By \eqref{eq:d(x,y)}, for $\mb{y} \in 
\tilde{H}(n,2)$, 
we have
\[
d(\mb{x}, \mb{y})^2
=-1-\frac{1}{n}+k_0^{(1)}-\frac{{k_0^{(1)}}^2}{n}
+ 2 l_1^{(1)}+4 l_2^{(1)} \in \{2,4\}.
\]
This implies that
\[
-1-\frac{1}{n}+k_0^{(1)}-\frac{{k_0^{(1)}}^2}{n}=0,
\]
and hence $n=k_0^{(1)}+1+2/(k_0^{(1)}-1)$. 
The possible pairs are $(n,k_0^{(1)})=(5,2),(5,3)$, namely, $\mb{x}$ is contained in  
\[
X_1=\left( \left(\frac{2}{5}^4, -\frac{3}{5} \right)^P, \frac{1}{5}^5\right)^P
\text{ or }
X_2=\left( \left(\frac{3}{5}^3, -\frac{2}{5}^2 \right)^P, \frac{1}{5}^5\right)^P.
\]
Considering the distance of each pair of elements of $X_1 \cup X_2$, the maximal $2$-distance set containing $\tilde{H}(n,2)$ is
\[
\left( \left(\frac{2}{5}^4, -\frac{3}{5} \right)^P, \frac{1}{5}^5\right)
\cup 
\left(  \frac{1}{5}^5,\left(\frac{3}{5}^3, -\frac{2}{5}^2 \right)^P\right) \cup \tilde{H}(n,2) \quad[\text{40 points}],
\]
up to block permutations.

\section{Conditions for the maximality of $\tilde{H}(n,m)$}\label{sec:4}
In this section, we give a necessary and  sufficient conditon  for
$\tilde{H}(n,m)$ to be  a maximal $m$-distance set in $\mathbb{R}^{n(m-1)}$. We also show a manner of finding vectors that can be added to
$\tilde{H}(n,m)$.  

Let $\mathfrak{x}_j$ denote the subset of $\R^{n}$ defined by
\[
\mathfrak{x}_j=\mathfrak{x}_j(k_0^{(j)},\ldots,k_{t_j}^{(j)})=\left(\frac{k_0^{(j)}}{n}^{k_1^{(j)}},
\left(\frac{k_0^{(j)}}{n}-1\right)^{k_2^{(j)}},\ldots,\left(\frac{k_0^{(j)}}{n}-t_j+1\right)^{k_{t_j}^{(j)}}\right)^P
\]
such that $k_0^{(j)}=1+\sum_{i=1}^{t_j}(i-1)k_i^{(j)}$. Let $\mathfrak{X}$ denote the subset of $\R^{mn}$ defined by 
\begin{equation}\label{eq:XX}
 \mathfrak{X}
=(\mathfrak{x}_1,\ldots, \mathfrak{x}_m). 
\end{equation}
The integers $k_i^{(j)}$ with $1\leq i \leq t_j$ and $1\leq j \leq m$ are called the parameters of  $\mathfrak{X}$.  
 If $\mb{x} \in \R^{mn}$ can
be added to $\tilde{H}(n,m)$, then $\mb{x}$ is in $\mathfrak{X}$ with some parameters $k_i^{(j)}$.
Moreover if some vector in $\mathfrak{X}$ can be added to $\tilde{H}(n,m)$, then so does every vector in 
$ \mathfrak{X}$. 
For given $\mb{x}\in \mathfrak{X}$,  
let 
\[
M_{\mathfrak{X}}=\max_{ \mb{y} \in \tilde{H}(n,m)} d(\mb{x},\mb{y})^2. 
\]
Note that $M_{\mathfrak{X}}$ does not depend on the choice of $\mb{x} \in \mathfrak{X}$. 
The value $M_\mathfrak{X}$ determine if $x\in \mathfrak{X}$ can be added to $\tilde{H}(n,m)$ . 
\begin{lemma} \label{lem:M<2m}
Each vector $\mb{x}\in \mathfrak{X}$ can be added to $\tilde{H}(n,m)$ while maintaining $m$-distance if and only if 
$M_\mathfrak{X}$ is an even positive integer less than or equal to $2m$. 
\end{lemma}
\begin{proof}
Sufficiency is clear. Suppose $M_\mathfrak{X}$ is an even positive integer less than or equal to $2m$. 
For any $\mb{x},\mb{x}' \in \mathfrak{X}$ and $\mb{y}, \mb{y}'\in \tilde{H}(n,m)$, the difference between $d(\mb{x},\mb{y})^2$ and $d(\mb{x}',\mb{y}')^2$ is an even integer by \eqref{eq:d(x,y)}. 
Therefore, $d(\mb{x},\mb{y})^2$ is an even positive integer less than or equal to $2m$ for any $\mb{x}\in \mathfrak{X}$ and $\mb{y}\in \tilde{H}(n,m)$. This implies each $\mb{x}\in \mathfrak{X}$ can be added to $\tilde{H}(n,m)$. 
\end{proof}
It follows from \eqref{eq:d(x,y)} that
\begin{equation}\label{eq:M_X}
M_\mathfrak{X} 
=
\sum_{j=1}^m\left(1-n-2k_0^{(j)}-\frac{{k_0^{(j)}}^2}{n}
+ \sum_{i=1}^{t_j}i^2 k_i^{(j)}
+ 2t_j
\right).
\end{equation}
If $t_l \geq 3$ for $\mathfrak{X}$ with parameters $k_i^{(j)}$,  
then we can obtain $\mathfrak{X}'$ with parameters ${k'}_i^{(j)}$, where ${k'}_i^{(j)}$ are defined as follows:  
\begin{align*}
&{k'}_i^{(j)}=k_i^{(j)}\qquad  \text{ for $(j,i)\ne (l,1),(l,2),(l,t_l-1),(l,t_l)$}, \\
&{k'}_1^{(l)}=k_1^{(l)}-1, \quad
{k'}_2^{(l)}=k_2^{(l)}+1, \quad
{k'}_
{t_l-1}^{(l)}=k_{t_l-1}^{(l)}+1, \quad
{k'}_{t_l}^{(l)}=k_{t_l}^{(l)}-1 
\quad \text{ if $t_l\geq 4$}, \\
&{k'}_1^{(l)}=k_1^{(l)}-1, \quad
{k'}_2^{(l)}=k_2^{(l)}+2, \quad
{k'}_{3}^{(l)}=k_{3}^{(l)}-1 
\quad \text{ if $t_l=3$.}
\end{align*} 
 
\begin{lemma}\label{lem:X'}
If $M_{\mathfrak{X}}$ is even, then 
$M_{\mathfrak{X}'}$ is also even and
\[
M_{\mathfrak{X}}>M_{\mathfrak{X}'}. 
\]
\end{lemma}
\begin{proof}
It follows from \eqref{eq:M_X} that
$M_{\mathfrak{X}}-M_{\mathfrak{X}'}
=2(2t_l-t_l'-2)$. 
Note that $t_l'=t_l$ if $k_{t_l}^{(l)}\geq 2$, and 
$t_l'=t_l-1$ if $k_{t_l}^{(l)}= 1$. 
Since $t_l\geq 3$ holds, $M_{\mathfrak{X}}-M_{\mathfrak{X}'}$ is an even positive integer. This lemma therefore follows.  
\end{proof}
\begin{remark}\label{rem:X_0}
By Lemmas \ref{lem:M<2m} and \ref{lem:X'}, if each element of $\mathfrak{X}$ can be added to $\tilde{H}(n,m)$, then so does each element of $\mathfrak{X}'$. 
Repeating this modification 
$\mathfrak{X}'$, we can obtain $\mathfrak{X}_0$,  which satisfies $t_j \leq 2$ for $1\leq j\leq m$ and each element of $\mathfrak{X}_0$ can be added to $\tilde{H}(n,m)$.  
\end{remark}
\begin{lemma} \label{lem:k}
The following are equivalent.  
\begin{enumerate}
\item $\tilde{H}(n,m)$ is not  
maximal as an $m$-distance set. 
\item There exist integers $l$, $k_0^{(1)},\ldots, k_0^{(m)}$ such that
$n\geq k_0^{(1)}\geq \cdots \geq k_0^{(l)}>1=k_0^{(l+1)}=\cdots= k_0^{(m)}$, $k_0^{(i)}\ne n$ for some $i$, and 
\begin{equation}
\sum_{j=1}^m  \frac{k_0^{(j)}(n-k_0^{(j)})}{n} 
+2l   \label{eq:M(k_0)}
\end{equation}
is an even integer less than or equal to $2m$.  
\end{enumerate}
\end{lemma}
\begin{proof}
Suppose $\mathfrak{X}$ with parameters $k_i^{(j)}$ satisfies $t_1=\cdots =t_l=2$ and $t_{l+1}=\cdots =t_m=1$ for some $l$. 
From \eqref{eq:M_X} 
and equations $k_1^{(j)}
=n+1-k_0^{(j)}$, $k_2^{(j)}=k_0^{(j)}-1$, we have
\begin{align*}
M_{\mathfrak{X}}
&=\sum_{j=1}^m\frac{k_0^{(j)}(n-k_0^{(j)})}{n} 
+2l.
\end{align*}
By Lemma~\ref{lem:M<2m} and Remark~\ref{rem:X_0}, this lemma follows. 
\end{proof}
\begin{remark} \label{rem:k}
By symmetry of \eqref{eq:M(k_0)}, there exist $l$, $k_0^{(1)},\ldots, k_0^{(m)}$ such that they satisfy the condition (2) 
in Lemma~\ref{lem:k} if and only if 
there exist $l$, $k_1,\ldots,k_m$ such that  $0\leq k_1 \leq \cdots \leq k_l\leq n/2$, 
 $k_{l+1}=\cdots=k_m=1$, $k_i\ne 0$ for some $i$, and  
\[
\sum_{j=1}^m \frac{k_j(n-k_j)}{n} +2l
\] is an even integer less 
than or equal to $2m$. 
\end{remark}
\begin{theorem} \label{thm:main}
If $\tilde{H}(n,m)$ is not maximal as an $m$-distance set for some $m$, then the maximum value of $n$ is  $m^2+m-1$.   
\end{theorem}
\begin{proof}
Suppose $\tilde{H}(n,m)$ is not maximal as an $m$-distance set.
By Lemma~\ref{lem:k} and Remark~\ref{rem:k}, 
there exist $l$, $k_1,\ldots, k_m$ such that 
 $0\leq k_1 \leq \cdots \leq k_l\leq n/2$, 
$k_{l+1}=\cdots=k_m=1$, $k_i\ne 0$ for some $i$, and   
\[
\sum_{j=1}^m k_j- \sum_{j=1}^m\frac{k_j^2}{n} +2l 
\] 
is an even integer at most $2m$.  
Let $M=\sum_{j=1}^m k_j- \sum_{j=1}^mk_j^2/n +2l$, and $t=\sum_{j=1}^mk_j^2/n$.  
If $l=0$ holds, then 
$M=m-m/n \in \mathbb{Z}
$. 
This implies $n \leq m < m^2+m-1$. 
 Therefore we may suppose $l\geq 1$. 

Assume $t\geq 2m-1$. Since $k_i \leq n/2$ for each $i$, we have  
$
M  \geq t+2l\geq 2m +1, 
$
which contradicts $M\leq 2m$. 

Assume $t \leq 2m-2$. Since $M\leq 2m$, we have
\[
\sum_{j=1}^l k_j \leq 
t+m-l. 
\]
It is therefore satisfied that
\begin{align*}
n&=\frac{1}{t} \sum_{j=1}^m k_j^2
=\frac{1}{t}\sum_{j=1}^l k_j^2 +
\frac{m-l}{t}
 \leq \frac{1}{t}\Big(
\sum_{j=1}^l k_j\Big)^2 +\frac{m-l}{t}
 \leq \frac{(t+m-l)^2}{t}+\frac{m-l}{t}\\
& \leq t+2(m-1)+ \frac{m(m-1)}{t} \leq
\max \{m^2+m-1, \frac{9}{2}m-4 \} =m^2+m-1 
\end{align*}
for $m\geq 2$ and  
$1 \leq t \leq 2m -2$. 
Moreover for $n=m^2+m-1$, $k_1=m$, and $k_2=\ldots=k_m=1$, 
we have $M=2m$. The theorem therefore follows.  
\end{proof}
The remaining part of this section relates to the smallest value of $m$ for fixed $n$, $l$, $k_0^{(1)},\ldots,k_0^{(l)}$ in Lemma~\ref{lem:k} (2).
Each vector in $\mathfrak{X}$ can be added to $\tilde{H}(n,m)$ if and only if each vector in $((1,0^{n-1})^P,\mathfrak{X})$ can be added to $\tilde{H}(n,m+1)$ by Lemma~\ref{lem:M<2m} and \eqref{eq:M_X}. Thus we may suppose $k_0^{(j)}<n$ 
for each $j\in\{1,\ldots,m\}$. Moreover the following holds. 
\begin{proposition}
Let $i$ be a positive integer.
Suppose $i$ is even if $n$ is even, and $i$ is arbitrary
if $n$ is odd.   
If each vector in $\mathfrak{X}
\subset \mathbb{R}^{nm}$ can be added to $\tilde{H}(n,m)$ while maintaining $m$-distance, then each vector in $(\mathfrak{X},1^n,\ldots, 1^n)\subset \mathbb{R}^{n(m+in)}$ can be added to $\tilde{H}(n,m+in)$ while maintaining $m$-distance. 
\end{proposition}
\begin{proof}
Let $k_i^{(j)}$ be the parameters of $\mathfrak{X}$. 
Since each vector in $\mathfrak{X}$ can be added to 
$\tilde{H}(n,m)$, the value 
\[
M_\mathfrak{X} 
=
\sum_{j=1}^m\left(1-n-2k_0^{(j)}-\frac{{k_0^{(j)}}^2}{n}
+ \sum_{i=1}^{t_j}i^2 k_i^{(j)}
+ 2t_j
\right)
\]
is an even integer at most $2m$. 
The set $\mathfrak{Y}=(\mathfrak{X},1^n,\ldots, 1^n)\subset \mathbb{R}^{n(m+in)}$ satisfies that
the value 
\[
M_\mathfrak{Y}
=M_{\mathfrak{X}}+i(n-1) 
\]
is an even integer at most $2(m+in)$. This implies the proposition. 
\end{proof}
The following theorem gives the minimum value of $m$ such that $\tilde{H}(n,m)$ is not maximal for 
fixed $n$, $l$, $k_0^{(1)},\ldots, k_0^{(l)}$ in 
Lemma~\ref{lem:k}.
\begin{theorem}
Let integers $n$, $l$, $k_0^{(1)},\ldots,k_0^{(l)}$  be fixed. Suppose $k_0^{(j)}<n$ for each $j\in \{1,\ldots ,l \}$.  
Let \[
i=\begin{cases}\mathfrak{i} \text{ if $n$ is odd,}\\
\mathfrak{i} \text{ if $n$ is even and $\mathfrak{i}$ is even,}\\
\mathfrak{i}+1
\text{ if $n$ is even and $\mathfrak{i}$ is odd,}
\end{cases}
\]
where  
\[
\mathfrak{i}=\left\lceil \frac{\sum_{j=1}^l k_0^{(j)}(1+k_0^{(j)})}{n+1} \right\rceil. 
\]
If integers $n$, $l$, $k_0^{(1)},\ldots,k_0^{(m)}$ satisfy Lemma~\ref{lem:k} (2), then the minimum value of $m$ is 
\[
l-\sum_{j=1}^l (k_0^{(j)})^2+in. 
\]
\end{theorem}
\begin{proof}
Since the value 
\[
M=\sum_{j=1}^lk_0^{(j)}+m+l-\frac{\sum_{j=1}^l (k_0^{(j)})^2+m-l}{n}
\]
is an integer, $\sum_{j=1}^l (k_0^{(j)})^2+m-l$ is a multiple of $n$. 
We may express $m=l-\sum_{j=1}^l(k_0^{(j)})^2+in$ for some integer $i$. Now we have 
\[
M=\sum_{j=1}^{l}k_0^{(j)}(1-k_0^{(j)})+(n-1)i+2l.
\]
Since $0<M\leq 2m 
=2(l-\sum_{j=1}^l(k_0^{(j)})^2+in)$, it follows that  
\[
i\geq \left\lceil\frac{\sum_{j=1}^l
k_0^{(j)}(1+k_0^{(j)})}{n+1}\right\rceil.
\]
When $n$ is odd, $M$ is even.  
When $n$ is even, $M$ is even if and only if $i$ is even.
Therefore the theorem follows from Lemma~\ref{lem:M<2m}.
\end{proof}
\section{Largest $m$-distance sets that contain $\tilde{H}(n,m)$}\label{sec:5}
In this section, we classify the 
largest $m$-distance sets that contain $\tilde{H}(n,m)$. For fixed $m$ and each $n$
such that $n \leq m^2+m-1$, we obtain all possible parameters $( k_0^{(1)},\ldots, k_0^{(m)})$ satisfying Lemma~\ref{lem:k} (2) by an exhaustive computer search. For $m=3,4$, Table \ref{table:6} shows the all sets $\mathfrak{X}_0$ obtained from the possible parameters $(k_0^{(1)},\ldots, k_0^{(m)})$,  up to block permutations. 

\begin{table}
\begin{center}
\caption{Each vector in $\mathfrak{X}_0$ can be added to $\tilde{H}(n,m)$} 
\label{table:6}
{\small
\begin{tabular}{|c|c|l|}
\hline
$m$ & $n$ & $n\mathfrak{X}_0$\\ \hline
$3$ & $3$ &  $( 1^3, 1^3, 1^3 )$, $( (2^2,-1)^P, 1^3, 1^3 )$, $( (2^2,-1)^P, (2^2,-1)^P, 1^3 )$ \\ \cline{2-3}
&$5$ & $( (5,0^4)^P, (2^4,-3^1)^P, 1^5 )$,
    $((5,0^4)^P, (3^3,-2^2)^P, 1^5 )$ \\ \cline{2-3}
&$9$ &  $((4^6,-5^3)^P, 1^9, 1^9 )$,
    $( (5^5,-4^4)^P, 1^9, 1^9 )$\\ \cline{2-3}
& $11$ &  $( (3^9,-8^2)^P, 1^{11}, 1^{11} )$,
    $( (8^4,-3^7)^P, 1^{11}, 1^{11} )$ \\ \cline{2-3}
\hline
$4$ & $2$ & $ (1^2, 1^2, 1^2, 1^2)$\\ \cline{2-3}
&$3$ & $( (3,0^2)^P, 1^3, 1^3, 1^3 )$,
     $( (3,0^2)^P, (2^2,-1)^P, 1^3, 1^3 )$,
     $( (3,0^2)^P, (2^2,-1)^P, (2^2,-1)^P, 1^3 )$ \\ \cline{2-3}
&$5$ &  $( (2^4,-3^1)^P, (2^4,-3^1)^P, 1^5, 1^5 )$,
     $( (3^3,-2^2)^P, (2^4,-3^1)^P, 1^5, 1^5 )$,\\ 
&&
     $( (3^3,-2^2)^P, (3^3,-2^2)^P, 1^5, 1^5 )$,
     $( (5,0^4)^P, (5,0^4)^P, (2^4,-3^1)^P, 1^5 )$,\\
&&
     $( (5,0^4)^P, (5,0^4)^P, (3^3,-2^2)^P, 1^5 )$ \\ \cline{2-3}
&$6$ & $( (3^4,-3^2)^P, 1^6, 1^6, 1^6 )$,
     $( (5^2,-1^4)^P, (3^4,-3^2)^P, 1^6, 1^6 )$\\\cline{2-3}
&$7$ & $( (2^6,-5^1)^P, 1^7, 1^7, 1^7 )$,
     $( (5^3,-2^4)^P, 1^7, 1^7, 1^7 )$,
     $( (6^2,-1^5)^P, (2^6,-5^1)^P, 1^7, 1^7 )$,\\
&&
     $( (6^2,-1^5)^P, (5^3,-2^4)^P, 1^7, 1^7 )$ \\\cline{2-3}
&$9$& $((9,0^8)^P, (4^6,-5^3)^P, 1^9, 1^9 )$,
    $( (9,0^8)^P, (5^5,-4^4)^P, 1^9, 1^9 $) \\\cline{2-3}
&$11$ &  $((11,0^{10})^P, (3^9,-8^2)^P, 1^{11}, 1^{11} )$,
    $( (11,0^{10})^P, (8^4,-3^7)^P, 1^{11}, 1^{11} )$\\\cline{2-3}
&$13$ & $( (6^8,-7^5)^P, 1^{13}, 1^{13}, 1^{13} )$,
    $( (7^7,-6^6)^P, 1^{13}, 1^{13}, 1^{13} )$\\\cline{2-3}
& $14$ &$( (5^{10},-9^{4})^P, 1^{14}, 1^{14}, 1^{14} )$,
    $( (9^{6},-5^{8})^P, 1^{14}, 1^{14}, 1^{14} )$ \\\cline{2-3}
& $19$ & $( (4^{16},-15^3)^P, 1^{19}, 1^{19}, 1^{19} )$,
    $((15^5,-4^{14})^P, 1^{19}, 1^{19}, 1^{19} )$\\
\hline
\end{tabular}
}
\end{center}
\end{table}

By repeating the inverse modification of $\mathfrak{X}'$ in Remark~\ref{rem:X_0},   $\mathfrak{X}_0$ in Table \ref{table:6} can be modified 
 to  
$\mathfrak{X}$ whose element can be added to $\tilde{H}(n,m)$.  
When $M_{\mathfrak{X}}<2m$, we apply the inverse modification of $\mathfrak{X}'$ to $\mathfrak{X}$. 
Note that the inverse modification of $\mathfrak{X}'$ sometimes has several  possibilities. 
The sets $\mathfrak{X}_0$ with $M_{\mathfrak{X}_0}<2m$ are in Table \ref{table:7} for $m=3,4$. 
\begin{table}
\caption{$\mathfrak{X}_0$ satisfying $M_{\mathfrak{X}_0}<2m$} \label{table:7}
\begin{center}
{\small
\begin{tabular}{|c|c|l|}
\hline
$m$ & $n$ & $n\mathfrak{X}_0$\\ \hline
$3$ & $3$ & $( 1^3, 1^3, 1^3 )$, $( (2^2,-1)^P, 1^3, 1^3 )$\\
\hline
$4$ & $2$ &$ (1^2, 1^2, 1^2, 1^2)$\\
&$3$ & $( (3,0^2)^P, 1^3, 1^3, 1^3 )$,
     $( (3,0^2)^P, (2^2,-1)^P, 1^3, 1^3 )$ \\
&$6$ & $( (3^4,-3^2)^P, 1^6, 1^6, 1^6 )$\\
&$7$ & $( (2^6,-5^1)^P, 1^7, 1^7, 1^7 )$,
     $( (5^3,-2^4)^P, 1^7, 1^7, 1^7 )$\\
\hline
\end{tabular}
}
\end{center}
\end{table}
 Table \ref{table:8} is the list of $\mathfrak{X}$ obtained from $\mathfrak{X}_0$, up to block permutations.

\begin{table}
\caption{$\mathfrak{X}$ obtained from $\mathfrak{X}_0$ in Table \ref{table:7}} \label{table:8}
\begin{center}
\begin{tabular}{|c|c|l|}
\hline
$m$& $n$ &  $n\mathfrak{X}$\\ \hline
$3$ & $3$& $((4,1,-2)^P,1^3,1^3)$, $((5,-1^2)^P,1^3,1^3)$ \\ \hline
$4$ & $2$& $((3,-1)^P,1^2,1^2,1^2)$\\ &$3$& $((3,0^2)^P,(4,1,-2)^P,1^3,1^3)$, $((3^2,-3)^P,1^3,1^3,1^3)$, \\&&
$((3,0^2)^P,(5,-1^2)^P,1^3,1^3)$ \\
&$6$ & $((9,3^2,-3^3)^P,1^6,1^6,1^6)$\\
&$7$ & $((9,2^4,-5^2)^P,1^7,1^7,1^7)$, 
$((12,5,-2^5)^P,1^7,1^7,1^7)$\\
\hline
\end{tabular}
\end{center}
\end{table}

We would like to find the largest sets, which can be added to $\tilde{H}(n,m)$ while maintaining $m$-distance and are in a union of the sets $\mathfrak{X}$ in Tables \ref{table:6} and \ref{table:8}.   
First we classify the largest subsets of $\mathfrak{X}$ whose maximum distance is at most $\sqrt{2m}$. Such subsets can be added to $\tilde{H}(n,m)$. 
The maximum distance of each $\mathfrak{X}$ in Table \ref{table:6} is at most $\sqrt{2m}$ except for the sets in Table \ref{table:9}. 
\begin{table}
\caption{$\mathfrak{X}(\{k_i^{(j)}\})$ which needs Erd\H{o}s--Ko--Rado type results} \label{table:9}
\begin{center}
{\small
\begin{tabular}{|c|c|c|c|c|c|}\hline
$m$ & $n$ & $n\mathfrak{X}$ & 
Maximum set $X$&$|X|$& reason\\ \hline
$3$ 
&$9$ &  
    $((5^5,-4^4)^P,1^9,1^9)$
&$(\mathscr{F}_{3}(5,2,5),1^9,1^9)$&$56$ & Theorem~\ref{thm:EKR} (2)\\
& $11$ &  
    $((8^4,-3^7)^P,1^{11},1^{11})$
& $(\mathscr{F}_{0}(4,1,8),1^{11},1^{11})$& $120$&Theorem~\ref{thm:EKR} (1)\\
\hline
$4$ 
&$7$ & 
 $((6^2,-1^5)^P,(5^3,-2^4)^P,1^7,1^7)$
&$((6^2,-1^5)^P,
\mathscr{F}_{0}(3,1,5),1^7,1^7 )$& $315$&Lemma~\ref{lem:exce}\\
&$9$& 
$((9,0^8)^P,(5^5,-4^4)^P,1^9,1^9)$
& $((9,0^8 )^P,
\mathscr{F}_{3}(5,2,5),1^9,1^9 )$ & $504$&Lemma~\ref{lem:exce}\\
&$11$ &  
$((11,0^{10})^P,(8^4,-3^7)^P,1^{11},1^{11})$    
&
$((11,0^{10})^P,
\mathscr{F}_{0}(4,1,8),1^{11},1^{11} )$& $1320$&Lemma~\ref{lem:exce}\\
&$13$ & 
$((6^8,-7^5)^P,1^{13},1^{13},1^{13})$
&$(\mathscr{F}_{4}(8,4,6),1^{13},1^{13},1^{13})$&$495$ &Theorem~\ref{thm:EKR} (3)\\
& & & $(\mathscr{F}_{5}(8,4,6),1^{13},1^{13},1^{13})$&$495$ &Theorem~\ref{thm:EKR} (3) \\
& &    
$((7^7,-6^6)^P,1^{13},1^{13},1^{13})$
&$(\mathscr{F}_{3}(7,3,7),1^{13},1^{13},1^{13})$&$372$ &Theorem~\ref{thm:EKR} (2)\\
& $14$ &
$((9^6,-5^8)^P,1^{14},1^{14},1^{14})$    
&$(\mathscr{F}_{1}(6,2,9),1^{14},1^{14},1^{14})$& $525$&Theorem~\ref{thm:EKR} (2)\\
& $19$ & 
$((15^5,-4^{14})^P,1^{19},1^{19},1^{19})$    
&$(\mathscr{F}_{0}(5,1,15),1^{19},1^{19},1^{19})$&$3060$ &Theorem~\ref{thm:EKR} (1)\\
\hline
\end{tabular}
}
\end{center}
\end{table}
The largest subsets of $\mathfrak{X}$ in Table~\ref{table:9} is at most $\sqrt{2m}$ can be determined by some results related to the Erd\H{o}s--Ko--Rado theorem. Note that $(\alpha^{k_1},(\alpha-1)^{k_2})^P$ ({\it resp.} $(\alpha^{k_1},(\alpha-1)^{k_2}, (\alpha-2)^{k_3})^P$) is isometric to $(1^{k_1},0^{k_2})^P$ ({\it resp.} $(1^{k_1},0^{k_2},-1^{k_3})^P$). 
A subset~$X$ of $(1^k,0^{n-k})$ 
is {\it $(k-m)$-intersecting} 
if $d(\mb{x},\mb{y})^2 \leq 2m$ for any 
$\mb{x},\mb{y} \in X$. 
A family of non-empty subsets $X_1,\ldots,X_t$ of $(1^k,0^{n-k})$ is  
 {\it cross-intersecting} 
if $d(\mb{x},\mb{y})^2 < 2k$ for any 
$\mb{x}\in X_i$, $\mb{y} \in X_j$ and any $i$, $j$ with $i\ne j$.
 Let $\mathscr{F}_r$ denote 
\[
\mathscr{F}_r=\mathscr{F}_r(k,t)=\{(x_1,\ldots,x_n) \in (1^k,0^{n-k})^P \colon\,
|\{i \in \{1,\ldots, t+2r \} \colon\, x_i=1\}| \geq t+r
 \}
\] 
for $n\geq k\geq t+r$ and $n \geq t+2r$. 
Note that 
$\mathscr{F}_0=(1^t,(1^{k-t},0^{n-k})^P)$. 
We collect Erd\H{o}s--Ko--Rado type results that are needed later. 
\begin{theorem}[{\cite{EKR61,W84,AK97}}] \label{thm:EKR}
If $X \subset (1^k,0^{n-k})$ is $t$-intersecting,   
then the following hold. 
\begin{enumerate}
\item If $n> (k-t+1)(t+1)$, then 
$|X| \leq\binom{n-t}{k-t}$, and the largest set 
is $\mathscr{F}_0$, up to permutations of coordinates.   
\item If $(k-t+1)(2+(t-1)/(r+1))<n<(k-t+1)(2+(t-1)/r)$ for some $r \in \mathbb{N}$, then $|X| \leq |\mathscr{F}_r|$, and the largest set is $\mathscr{F}_r$, up to permutations of coordinates.
\item If $(k-t+1)(2+(t-1)/(r+1))=n$ for some $r \in \mathbb{N}$, then $|X| \leq |\mathscr{F}_r|=|\mathscr{F}_{r+1}|$, and the largest set is $\mathscr{F}_r$ or $\mathscr{F}_{r+1}$, up to permutations of coordinates.
\end{enumerate}
\end{theorem} 

\begin{theorem}[\cite{HM67}]
\label{thm:sum}
Suppose $n \geq 2k$. If a pair of subsets $X, Y$ of $(1^k,0^{n-k})$ is cross-intersecting, 
then 
\[
|X|+|Y| \leq \binom{n}{k}-\binom{n-k}{k}+1.
\]
\end{theorem}

\begin{theorem}[{\cite{H77,B14}}] \label{lem:sum3}
Suppose $n>2k$ and $s> n/k$. If a family of subsets $X_1,\ldots, X_s$ of  $(1^k,0^{n-k})$ is  cross-intersecting,  
then  
\[
\sum_{i=1}^s|X_i| \leq s \binom{n-1}{k-1}. 
\]
If equality holds, then $X_i=\mathscr{F}_0$ for each $i \in \{1,\ldots, s\}$, up to permutations of coordinates.  
\end{theorem}

We use the notation 
\[
\mathscr{F}_{r}(k,t,k_0)=\{\mb{x} \in (k_0^k,(k_0-n)^{n-k})^P   \colon\, 
1/n(\mb{x}- (k_0-n)\mb{1})  \in \mathscr{F}_{r}(k,t)\},
\]
where $\mb{1}=(1^n)$. 
By Theorem \ref{thm:EKR}, we can determine the largest subsets as Table \ref{table:9} except for $(m,n)=(4,7),(4,9),(4,11)$. 
\begin{lemma} \label{lem:exce}
If $X$ is the largest subset of $\mathfrak{X}_0$ whose distances are in $\{\sqrt{2},\sqrt{4},\sqrt{6},\sqrt{8}\}$, 
 then the following hold.
\begin{enumerate}
\item For 
$\mathfrak{X}_0=
(1/7)((6^2,-1^5)^P,
(5^3,-2^4)^P,
1^7,1^7)$, 
we have
\[
X=\left(\left(\frac{6}{7}^2,-\frac{1}{7}^5 \right)^P,
\frac{1}{7}\mathscr{F}_{0}(3,1,5),\frac{1}{7}^7,\frac{1}{7}^7  \right),
\]
up to permutations on coordinates in the second block.
\item For 
$\mathfrak{X}_0=
(1/9)((9,0^8)^P,(5^5,-4^4)^P,
1^9,1^9)
$, 
we have
\[
X=\left(\left(1^1,0^8 \right)^P,
\frac{1}{9}\mathscr{F}_{3}(5,2,5),\frac{1}{9}^9,\frac{1}{9}^9  \right),
\]
up to permutations on coordinates in the second block.
\item For 
$\mathfrak{X}_0=
(1/11)((11,0^{10})^P,
(8^4,-3^7)^P,
1^{11},1^{11})
$, 
we have
\[
X=\left(\left(1^1,0^{10} \right)^P,
\frac{1}{11}\mathscr{F}_{0}(4,1,8),\frac{1}{11}^{11},\frac{1}{11}^{11}  \right),
\]
up to permutations on coordinates in the second block.
\end{enumerate}
\end{lemma}
\begin{proof}
We use the notation $
S_{\mb{a}}=\{\mb{x}_2  \mid \mb{x}_1= \mb{a}, (\mb{x}_1,\mb{x}_2,\mb{x}_3,\mb{x}_4)\in X\}$, 
where $\mb{x}_i$ is a vector in $\mathbb{R}^n$.  

(1): 
If $d(\mb{x}_1,\mb{y}_1)^2=4$, then $d(\mb{x}_2,\mb{y}_2)^2\leq 4$ for $(\mb{x}_1,\mb{x}_2,\mb{x}_3,\mb{x}_4),(\mb{y}_1,\mb{y}_2,\mb{y}_3,\mb{y}_4) \in X$.  
Thus we
have $d(\mb{a},\mb{b})^2\leq 4$ for any $\mb{a} \in S_{\mb{x}_1}$, $\mb{b} \in S_{\mb{y}_1}$ such that $d(\mb{x}_1,\mb{y}_1)=4$,
and hence a pair $\{S_{\mb{x}_1}, S_{\mb{y}_1}\}$ is cross-intersecting. 
The set $((6/7)^2,(-1/7)^5)^P$ is isometric to $(1^2,0^5)^P$. The set $(1^2,0^5)^P$ has a  triangle decomposition $\{T_i\}_{0\leq i\leq 6}$, for example
\[
T_i=\{(0,0,1,0,1,0,0)^{\sigma(i)},(0,1,0,0,0,1,0)^{\sigma(i)},(1,0,0,0,0,0,1)^{\sigma(i)}\},
\]
where $(x_1,\ldots,x_7)^{\sigma(i)}=(x_{1+i},\ldots,x_{7+i})$ such that the indices are in $\mathbb{Z}/7\mathbb{Z}$. 
By Theorems~\ref{thm:sum} and 
\ref{lem:sum3}, we have 
\[
\sum_{\mb{a} \in T_i}|S_{\mb{a}}|
\leq \begin{cases}
\binom{7}{3}=35, \text{ if two sets are empty,}\\
\binom{7}{3}-\binom{4}{3}+1=32, \text{ if only one set is empty,}\\
3\binom{6}{2}=45, \text{ if no set is empty.}
\end{cases}
\]
It therefore follows that
\[
|X|=\sum_{i=0}^6\sum_{\mb{a} \in T_i} |S_{\mb{a}}| \leq 7 \cdot 45=315. 
\]
If equality holds, then $X$ is the set defined in (1).

(2): 
The 
second block can be identified with
 $(1^4,0^5)^P$, which
is isometric to $(1^5,0^4)^P$. 
For distinct vectors $(\mb{x}_1,\mb{x}_2,\mb{x}_3,\mb{x}_4),(\mb{y}_1,\mb{y}_2,\mb{y}_3,\mb{y}_4)\in X$, 
we have $d(\mb{x}_1,\mb{y}_1)^2=2$ and 
 $d(\mb{x}_2,\mb{y}_2)^2 \leq 6$. 
Thus a family $\{S_{\mb{x}_1}\}_{\mb{x}\in X}$ can be interpreted as a  
 cross-intersecting family. 
 The statement (2) therefore follows from Theorem~\ref{lem:sum3}.

(3): This case is proved by a similar manner to (2). 
\end{proof}

For $\mathfrak{X}$ 
in Table~\ref{table:8}, the largest subsets of $\mathfrak{X}$ whose maximum distance is at most $\sqrt{2m}$ are in Table \ref{table:10}. By computer search using Magma, we classify maximal cliques in the graph $(\mathfrak{X},E)$, 
where $E=\{(\mb{x},\mb{y}) \mid d(\mb{x},\mb{y})^2 \leq 2m\}$ except for $n=7$.   
We use the following results for $n=7$.
\begin{proposition}[\cite{ANX}]
Assume $m+k< l$. 
If $X\subset (1^m,0^k,-1^l)^P$ has maximum 
distance smaller than that of $(1^m,0^k,-1^l)^P$, then
\[
|X|\leq \binom{n-1}{m+k-1}\binom{m+k}{m}. 
\] 
The largest set is $(1,(1^{m-1},0^k,-1^l)^P) \cup (0,(1^{m},0^{k-1},-1^l)^P)$, up to permutations of coordinates. 
\end{proposition}
Let $X_1=(1,(0,-1^2)^P)\cup (0,(1,-1^2)^P)$, 
$Y_1=(-1,(1,0,-1)^P)$, and 
$Z_2=(-1,(1,0^2,-1)^P)$. We inductively define
\begin{align*}
X_k&=(0,X_{k-1})\cup (1,(0^k,-1^2)^P)\quad (k\geq 2),\\
Y_k&=(0,Y_{k-1})\cup (1,(0^k,-1^2)^P)\quad (k\geq 2),\\
Z_k&=(0,Z_{k-1})\cup (1,(0^k,-1^2)^P)\quad (k\geq 3).
\end{align*}

\begin{theorem}[\cite{ANX}]
If $X\subset (1,0^k,-1^2)^P$ has maximum 
distance smaller than that of $(1,0^k,-1^2)^P$, then
\[
|X|\leq \binom{k+3}{3}+2. 
\] 
The largest sets are $X_k$ $($$k\geq 1$$)$, $Y_k$ $($$k\geq 1$$)$, and $Z_k$ $($$k\geq 2$$)$, up to permutations of coordinates. 
\end{theorem}
\begin{table}
{\small 
\caption{Largest subset whose distance at most $\sqrt{2m}$ } \label{table:10}
\begin{center}
\begin{tabular}{|c|c|c|c|c|}
\hline
$m$& $n$ &  $n\mathfrak{X}$&largest subset $X$ &$|X|$ \\ \hline
$3$ & $3$& $((4,1,-2)^P,1^3,1^3)$&
$\{((u_1,1,v_1),1^3,1^3),((u_2,v_2,1),1^3,1^3),$ & \\
&&& $((1,u_3,v_3),1^3,1^3)\}$ $((u_i,v_i)=(4,-2) \text{ or } (-2,4))$&$3$  \\ \cline{3-5}
    &  & $((5,-1^2)^P,1^3,1^3)$& $\{((5,-1^2),1^3,1^3)\}$&$1$ \\ \hline
$4$ & $2$& $((3,-1)^P,1^2,1^2,1^2)$&$((3,-1)^P,1^2,1^2,1^2)$ &$2$ \\ 
\cline{2-5}
&$3$& $((3,0^2)^P,(4,1,-2)^P,1^3,1^3)$&
$\{((3,0^2)^P,(u_1,1,v_1),1^3,1^3),((3,0^2)^P,(u_2,v_2,1),1^3,1^3),$ & \\
&&& $((3,0^2)^P,(1,u_3,v_3),1^3,1^3)\}$ $((u_i,v_i)=(4,-2) \text{ or } (-2,4))$&$9$ \\ 
\cline{3-5}
&&$((3^2,-3)^P,1^3,1^3,1^3)$&$((3^2,-3)^P,1^3,1^3,1^3)$ & 3 \\ \cline{3-5}
&&$((3,0^2)^P,(5,-1^2)^P,1^3,1^3)$&
$((3,0^2),(5,-1^2)^P,1^3,1^3)$, $((3,0^2)^P,(5,-1^2),1^3,1^3)$ &$3$ 
\\ \cline{2-5}
&$6$ & $((9,3^2,-3^3)^P,1^6,1^6,1^6)$& 
$(9,(3^2,-3^3)^P)\cup(3^2,(9,-3^3)^P)\cup
(3,9,(3,-3^3)^P)$
&$18$\\ \cline{2-5}
&$7$ & $((9,2^4,-5^2)^P,1^7,1^7,1^7)$& $X_7(9)$, $Y_7(9)$, $Z_7(9)$&$37$\\ 
\cline{3-5}
&&$((12,5,-2^5)^P,1^7,1^7,1^7)$ &$(12,(5,-2^5)^P) \cup (5,(12,-2^5)^P)$&$12$ \\
\hline
\end{tabular}
\end{center}
}
\end{table}
Let $X_k(k_0)$ be the subset of $(k_0,(k_0-n)^k,(k_0-2n)^2)$ obtained from $X_k$ by replacing the entries $1$, $0$, $-1$ to $k_0$, $k_0-n$, $k_0-2n$, respectively. The sets $Y_k(k_0)$, $Z_k(k_0)$ are similarly defined. 

Let $\mathfrak{X}$ and $\mathfrak{Y}$
be distinct sets in Tables~\ref{table:6} and \ref{table:8}.
We consider when
 $\mb{x}\in \mathfrak{X}$,  $\mb{y}\in \mathfrak{Y}$  can be   
simultaneously added to $\tilde{H}(n,m)$.
An element $(x_{11},\ldots, x_{1n}, \ldots,
x_{m1},\ldots,x_{mn})$ of $\mathfrak{X}$ 
is {\it canonical} if 
$x_{i,l}\geq x_{i,l+1}$ for any $i\in\{1,\ldots,m\}$, $l \in \{1,\ldots,n-1\}$. 

\begin{lemma}\label{lem:5.1}
If  $\mathfrak{X}$
and $\mathfrak{Y}$ are distinct sets that are 
expressed as \eqref{eq:XX},   
then the following are equivalent.
\begin{enumerate}
\item There exist $\mb{x}\in \mathfrak{X}$,  $\mb{y}\in \mathfrak{Y}$
such that $d(\mb{x},\mb{y})^2 \in 
\{2,4,\ldots,2m\}$. 
\item For canonical elements $\mb{x}'\in \mathfrak{X}$,  $\mb{y}'\in \mathfrak{Y}$, we have $d(\mb{x}',\mb{y}')^2 \in 
\{2,4,\ldots,2m\}$.
\end{enumerate}
Moreover if 
\[
\max_{\mb{x}\in \mathfrak{X},  \mb{y}\in \mathfrak{Y}} d(\mb{x},\mb{y})^2 \in 
\{2,4,\ldots,2m\},
\]
then any $\mb{x}\in \mathfrak{X}$,  $\mb{y}\in \mathfrak{Y}$ satisfy $d(\mb{x},\mb{y})^2 \in 
\{2,4,\ldots,2m\}$. 
\end{lemma}
\begin{proof}
$(2)\Rightarrow (1)$ is clear. We prove $(1)\Rightarrow (2)$. 
Suppose $d(\mb{x},\mb{y})^2 \in \{2,4,\ldots, 2m\}$ for some $\mb{x} \in \mathfrak{X}$, $\mb{y} \in \mathfrak{Y}$. 
By permutation of coordinates, we may suppose $\mb{x}$ is canonical. When there exist $i$, $l$ such that $y_{i,l} < y_{i,l+1}$, we switch the positions of $y_{i,l}$, $y_{i,l+1}$ to reduce the distance between $\mb{x}$ and $\mb{y}$.
 Indeed, letting $x_{i,l}=\alpha$, $x_{i,l+1}=\alpha-s$, 
$y_{i,l}=\beta$, $y_{i,l+1}=\beta+t$, where $s$, $t$
are non-negative integers, the difference of the squared distances is
\begin{equation} \label{eq:pr1}
\big((x_{i,l}-y_{i,l})^2+(x_{i,l+1}-y_{i,l+1})^2\big)
-\big((x_{i,l}-y_{i,l+1})^2+(x_{i,l+1}-y_{i,l})^2\big)= 2st \geq 0. 
\end{equation}
By repeating this modification, we can obtain canonical element $\mb{y}$ that satisfies $d(\mb{x},\mb{y})^2 \in \{2,4,\ldots, 2m\}$. 

The second assertion is clear by \eqref{eq:pr1}. 
\end{proof}

Let $V(n,m)$ be the family consisting of all sets $\mathfrak{X}$ whose element can be added to $\tilde{H}(n,m)$. Let 
 $E(n,m)=\{(\mathfrak{X},\mathfrak{Y}) \colon\, d(x',y') \in \{2,4,\ldots,2m \} \}\subset V(n,m) \times V(n,m)$, where $x'$, $y'$ are the canonical elements of $\mathfrak{X}$, $\mathfrak{Y}$, respectively. 
 Let $\mathcal{G}(n,m)$ be the graph
$(V(n,m),E(n,m))$.  
By computer search using Magma, we can classify  maximal cliques in $\mathcal{G}(n,m)$, up to block permutations. 
Table~\ref{table:a} shows the all maximal cliques  in $\mathcal{G}(n,m)$. 
By Lemma~\ref{lem:5.1}, 
for any distinct sets $\mathfrak{X}$, $\mathfrak{Y}$ in a clique in Table~\ref{table:a},  
any $\mb{x} \in \mathfrak{X}$, $\mb{y} \in \mathfrak{Y}$ 
can be simultaneously added to $\tilde{H}(n,m)$,
except for the combination of $((3,0^2)^P,1^3,1^3,(4,1,-2)^P)$ and 
$(1^3,(3,0^2)^P,1^3,(5,-1^2)^P)$. 
\begin{table} 
\caption{Maximal cliques in $\mathcal{G}(n,m)$}
\label{table:a}
{\small
\begin{center}
\begin{tabular}{|c|c|c|}
\hline
$m$ & $n$ & maximal clique \\
\hline
$3$ & $3$ &   $(
            ( 5, -1^2 )^P,
             1^3 ,
             1^3 
        )$,
        $(
             1^3,
            (2^2, -1)^P,
            (2^2, -1)^P
        )$,
        $(
            (2^2, -1)^P,
            1^3,
            1^3
        )
    $ \\
\cline{3-3}
&&$(
        1^3,
        1^3,
        1^3
    )$, $(
        ( 4, 1, -2 )^P,
         1^3,
         1^3
    )$, $(
        1^3,
        ( 4, 1, -2 )^P,
        1^3
    )$, $(
         1^3,
         1^3,
        ( 4, 1, -2 )^P
    )$ \\
\cline{2-3}
& $5$& $(
            ( 3^3,  -2^2)^P,
             1^5,
            ( 5, 0^4)^P
        )$,
        $(
             1^5, 
             (2^4, -3)^P ,
            ( 5, 0^4)^P
        )$ \\
\cline{2-3}
&$9$& $(
        ( 5^5, -4^4)^P,
         1^9,1^9
    )$ \\
\cline{3-3}
&& $ (
        ( 4^6, -5^3)^P,
        1^9,1^9
    )$, $ (
        1^9,( 4^6, -5^3)^P,
        1^9
    )$, $ (
        1^9,( 4^6, -5^3)^P,
        1^9
    )$\\
\cline{2-3}
&$11$&$((3^9,-8^2)^P,1^{11},1^{11})$ \\
\cline{3-3} 
& &$((8^4,-3^7)^P,1^{11},1^{11})$ \\
\hline
$4$ &$2$ & $(1^2,1^2,1^2,1^2)$, $((3,-1)^P,1^2,1^2,1^2)$, $(1^2,(3,-1)^P,1^2,1^2)$,\\&&
$(1^2,1^2,(3,-1)^P,1^2)$, $(1^2,1^2,1^2,(3,-1)^P)$\\
\cline{2-3}
& $3$&  
$( ( 3, 0^2)^P, 1^3,1^3,1^3 )$,
$( 1^3, ( 3, 0^2)^P, 1^3,  ( 2^2,  -1 )^P)$,\\&&
$( 1^3, ( 2^2,-1 )^P, ( 3, 0^2 )^P,1^3 )$,
$( 1^3,1^3,(2^2, -1)^P,  (3, 0^2)^P )$,\\&&
$( ( 2^2, -1)^P, ( 3, 0^2)^P,(2^2,  -1)^P,  1^3)$, 
$( ( 2^2, -1 )^P, 1^3, ( 3, 0^2)^P,( 2^2,  -1)^P)$,\\&& 
$( ( 2^2, -1 )^P, (2^2,  -1)^P,1^3,( 3, 0^2)^P)$, 
$(( 3^2,-3 )^P, 1^3,1^3, 1^3 )$, \\&&
$( (3, 0^2)^P,( 4, 1, -2 )^P,1^3,1^3 )$, 
$( 1^3,( 5, -1^2)^P, ( 3, 0^2)^P,1^3)$, \\&&
$(   ( 3, 0^2)^P, 1^3,( 4, 1, -2 )^P,1^3 )$, 
$( 1^3,1^3,( 5, -1^2)^P, ( 3, 0^2)^P )$, \\&&  
$( ( 3, 0^2)^P,  1^3,1^3,( 4, 1, -2 )^P  )$,
$( 1^3, (3, 0^2)^P, 1^3, ( 5, -1^2)^P )$ \\
\cline{2-3}
&$5$ &
$(1^5,1^5,(2^4,-3)^P,(2^4,-3)^P)$, $((3^3,-2^2)^P,1^5,1^5,(2^4,-3)^P)$,\\
&& $(1^5,(3^3,-2^2)^P,(2^4,-3)^P,1^5)$, 
$((3^3,-2^2)^P,(3^3,-2^2)^P,1^5,1^5)$, \\
&& $((5,0^4)^P,1^5,(5,0^4)^P,(2^4,-3)^P)$, 
$(1^5,(5,0^4)^P,(2^4,-3)^P,(5,0^4)^P)$, \\
&& $((5,0^4)^P,(3^3,-2^2)^P,(5,0^4)^P,1^5)$,
$((3^3,-2^2)^P,(5,0^4)^P,1^5,(5,0^4)^P)$ \\
\cline{2-3}
&$6$& $((3^4,-3^2)^P,1^6,1^6,1^6)$, 
$(1^6,(5^2,-1^4)^P,1^6,(3^4,-3^2)^P)$,\\
&&$(1^6,(3^4,-3^2)^P,(5^2,-1^4)^P,1^6)$,
$(1^6,1^6,(3^4,-3^2)^P,(5^2,-1^4)^P)$,\\
&&$((9,3^2,-3^3)^P,1^6,1^6,1^6)$\\
\cline{2-3}
&$7$&$((2^6,-5)^P,1^7,1^7,1^7)$, 
$(1^7, (6^2,-1^5)^P,(5^3,-2^4)^P,1^7)$,\\
&& $(1^7,1^7, (6^2,-1^5)^P,(5^3,-2^4)^P)$, 
$(1^7,(5^3,-2^4)^P,1^7, (6^2,-1^5)^P)$,\\
&& $((9,2^4,-5^2)^P,1^7,1^7,1^7)$,\\
\cline{3-3}
&& $((5^3,-2^4)^P,1^7,1^7,1^7)$, 
$(1^7,(6^2,-1^5)^P,(2^6,-5)^P,1^7)$,\\
&&$(1^7,1^7,(6^2,-1^5)^P,(2^6,-5)^P)$, 
$(1^7,(2^6,-5)^P,1^7,(6^2,-1^5)^P)$,\\
&&$((12,5,-2^5)^P,1^7,1^7,1^7)$\\
\cline{2-3}
&$9$& $((9,0^8)^P,(5^5,-4^4)^P,1^9,1^9)$, 
$(1^9,1^9,(9,0^8)^P,(5^5,-4^4)^P)$\\
\cline{3-3}
&&$((9,0^8)^P,(4^6,-5^3)^P,1^9,1^9)$, 
$((9,0^8)^P,1^9,(4^6,-5^3)^P,1^9)$, \\
&&$((9,0^8)^P,1^9,1^9,(4^6,-5^3)^P)$ \\
\cline{2-3}
&$11$& $((11,0^{10})^P,(8^4,-3^7)^P,1^{11},1^{11})$, 
$(1^{11},1^{11},(11,0^{10})^P,(3^9,-8^2)^P)$\\
\cline{2-3}
&$13$&$((6^8,-7^5)^P,1^{13},1^{13},1^{13})$ \\
\cline{3-3}
& & $((7^7,-6^6)^P,1^{13},1^{13},1^{13})$\\
\cline{2-3}
&$14$&$((5^{10},-9^4)^P,1^{14},1^{14},1^{14})$ \\
\cline{3-3}
& & $((9^6,-5^8)^P,1^{14},1^{14},1^{14})$\\
\cline{2-3}
&$19$&$((4^{16},-15^3)^P,1^{19},1^{19},1^{19})$ \\
\cline{3-3}
& & $((15^5,-4^{19})^P,1^{19},1^{19},1^{19})$\\
\hline
\end{tabular}
\end{center}
}
\end{table}
We replace $\mathfrak{X}$ in a clique in Table \ref{table:a} with one of the largest subsets of $\mathfrak{X}$, which can be added to $\tilde{H}(n,m)$. Finally we obtain Table~\ref{table:12},  which includes the all largest subsets that can be added to $\tilde{H}(n,m)$.    
\begin{table}
\caption{Sets obtained from a maximal clique in $\mathcal{G}(n,m)$}
\label{table:12}
{\small
\begin{center}
\begin{tabular}{|c|c|c|c|c|}
\hline
$m$ & $n$ & $X$ & $|X|$ & 
$M$\\
\hline
$3$ & $3$  & $(
            ( 5, -1^2 ),
             1^3 ,
             1^3 
        ) \cup
        (
             1^3,
            (2^2, -1)^P,
            (2^2, -1)^P
        )\cup 
        (
            (2^2, -1)^P,
            1^3,
            1^3
        )
    $& $13$ & $40$\\
\cline{3-5}
&&$(
        1^3,
        1^3,
        1^3
    )\cup (
        ( 4, 1, -2 )^C,
         1^3,
         1^3
    )^P$& $10$ & $37$\\
\cline{2-5} 
&$5$& $(
            ( 3^3,  -2^2)^P,
             1^5,
            ( 5, 0^4)^P
        )
        \cup(
             1^5, 
             (2^4, -3)^P ,
            ( 5, 0^4)^P
        )$ & $75$& $200$ 
\\ \cline{2-5}
& $9$ & $(\mathscr{F}_3(5,2,5),1^9,1^9)$ &
$84$ & $785$ \\
\cline{3-5}
&& $ (
        ( 4^6, -5^3)^P,
        1^9,1^9
    )^P$
& $252$
& $981$
\\
\cline{2-5}
&$11$& $((3^9,-8^2)^P,1^{11},1^{11})$ &
$55$ 
&$1386$
\\ \cline{3-5}
& & $(\mathscr{F}_0(4,1,8),1^{11},1^{11})$
&$120$
&$1451$
\\ \hline
$4$ & $2$ & $(1^2,1^2,1^2,1^2)\cup((3,-1)^P,1^2,1^2,1^2)^P$ &$9$& $25$\\
\cline{2-5}
& $3$& $( ( 3, 0^2)^P, 1^3,1^3,1^3 )\cup ( 1^3, (( 3, 0^2)^P, 1^3,  ( 2^2,  -1 )^P)^C)$ & & \\ 
& & $\cup( ( 2^2, -1)^P, (( 3, 0^2)^P,(2^2,  -1)^P,  1^3)^C)\cup( 3^2,-3 )^P, 1^3,1^3, 1^3 )$ & & \\
& & $\cup( (3, 0^2)^P,(\{( -2,4, 1),( -2, 1, 4 )\},1^3,1^3)^C )\cup
( 1^3,(( 5, -1^2), ( 3, 0^2)^P,1^3)^C)$&$141$ & $222$ \\
\cline{3-5}
 & &$( ( 3, 0^2)^P, 1^3,1^3,1^3 )\cup ( 1^3, (( 3, 0^2)^P, 1^3,  ( 2^2,  -1 )^P)^C)$ & & \\ 
& & $\cup( ( 2^2, -1)^P, (( 3, 0^2)^P,(2^2,  -1)^P,  1^3)^C)\cup( 3^2,-3 )^P, 1^3,1^3, 1^3 )$ & & \\
& & $\cup( (3, 0^2)^P,(( 4,-2, 1)^C,1^3,1^3)^C )$&$141$ & $222$ \\
\cline{2-5}
&$5$&
$(1^5,1^5,(2^4,-3)^P,(2^4,-3)^P)\cup((3^3,-2^2)^P,1^5,1^5,(2^4,-3)^P)^{(1\, 2)(3\, 4)}$ & & \\
&& $\cup((3^3,-2^2)^P,(3^3,-2^2)^P,1^5,1^5)
\cup((5,0^4)^P,1^5,(5,0^4)^P,(2^4,-3)^P)^{(1\, 2)(3\, 4)} $ &&\\
&& $\cup ((5,0^4)^P,(3^3,-2^2)^P,(5,0^4)^P,1^5)^{(1\, 2)(3\, 4)}$&$975$&$1600$\\
\cline{2-5}
&$6$&
$((3^4,-3^2)^P,1^6,1^6,1^6)\cup  
(1^6,((5^2,-1^4)^P,1^6,(3^4,-3^2)^P)^C)$ &&\\
&&$\cup((9,(3^2,-3^3)^P)\cup(3^2,(9,-3^3)^C)\cup
(3,9,(3,-3^3)^C),1^6,1^6,1^6)$&$708$ &$2004$\\
\cline{2-5}
&$7$&$((2^6,-5)^P,1^7,1^7,1^7)\cup  
(1^7, ((6^2,-1^5)^P,\mathscr{F}_0(3,1,5),1^7)^C)$&&\\
&& $\cup (\{X_7(9), Y_7(9)\text{ or } Z_7(9)\},1^7,1^7,1^7)$ &$989$ & $3390$\\
\cline{3-5}
&& $((5^3,-2^4)^P,1^7,1^7,1^7)\cup(1^7,((6^2,-1^5)^P,(2^6,-5)^P,1^7)^C)$ &&\\
&&$\cup (12,(5,-2^5)^P)\cup (5,(12,-2^5)^P)$ &$488$ & $2889$\\
\cline{2-5}
&$9$& $((9,0^8)^P,\mathscr{F}_3(5,2,5),1^9,1^9)$, 
$(1^9,1^9,(9,0^8)^P,\mathscr{F}_3(5,2,5))$ &$1008$ &$7569$\\
\cline{3-5}
& &$((9,0^8)^P,((4^6,-5^3)^P,1^9,1^9)^C)$ &$2268$ &$8829$
\\
\cline{2-5}
&$11$& $((11,0^{10})^P,\mathscr{F}_0(4,1,8),1^{11},1^{11})\cup (1^{11},1^{11},(11,0^{10})^P,(3^9,-8^2)^P)$&$1925$&$16566$\\
\cline{2-5}
&$13$&$(\mathscr{F}_4(8,4,6),1^{13},1^{13},1^{13})$ &$495$& $29056$\\
\cline{3-5}
& & $(\mathscr{F}_3(7,3,7),1^{13},1^{13},1^{13})$ &$372$ &
$28933$\\
\cline{2-5}
&$14$&$((5^{10},-9^4)^P,1^{14},1^{14},1^{14})$ &$1001$ & $39417$\\
\cline{3-5}
& & $(\mathscr{F}_1(6,2,9),1^{14},1^{14},1^{14})$&$525$&$38941$\\
\cline{2-5}
&$19$&$((4^{16},-15^3)^P,1^{19},1^{19},1^{19})$ &$969$ &$131290$\\
\cline{3-5}
& & $(\mathscr{F}_0(5,1,15),1^{19},1^{19},1^{19})$ &$3060$ &$133381$\\
\hline
\end{tabular}
\end{center}
$(x_1,\ldots,x_n)^C=
\{(x_{1+i},\ldots,x_{n+i})\mid i\in \mathbb{Z}/n\mathbb{Z} \}$, 
$(X_1,\ldots,X_n)^C=
\bigcup_{i\in \mathbb{Z}/n\mathbb{Z}}(X_{1+i},\ldots,X_{n+i})$, 
$(X_1,X_2,X_3,X_4)^{(1\,2)(3\,4)}=
(X_1,X_2,X_3,X_4) \cup 
(X_2,X_1,X_4,X_3).$ 
}
\end{table}
\section{Maximal $2$-distance sets that are in $\R^{2n-1}$ and contain $\tilde{H}(n,2)$}\label{sec:6}
The Hamming graph $H(n,2)$ is embedded into 
$\R^{2(n-1)}$ as a $2$-distance set. Indeed the representation is 
\[
\tilde{H}(n,2)=((1^1,0^{n-1})^P,(1^1,0^{n-1})^P) \subset 
\{\mathbf{x} \in \R^{2n} \mid \sum_{i=1}^n x_i =1, \sum_{i=n+1}^{2n} x_i =1 \} \cong \R^{2(n-1)}, 
\]
where $\cong$ means isometry. 
In this section, we identify $\tilde{H}(n,2)$ with a $2$-distance set in $\R^{2n-1}$ as follows
\begin{align*}
\tilde{H}(n,2) \cong \hat{H}(n,2)&=((1^1,0^{n-1})^P,(1^1,0^{n-1})^P,0) \\
&\subset 
\{\mathbf{x} \in \R^{2n+1} \mid \sum_{i=1}^n x_i =1, \sum_{i=n+1}^{2n} x_i =1 \} \cong \R^{2n-1}. 
\end{align*}
We consider 
maximal $2$-distance sets that are in $\R^{2n-1}$ and  contain $\hat{H}(n,2)$.
Bannai, Sato, and Shigezumi~\cite{BSS12} considered a similar problem for the Johnson graph $J(n,2)$, and classified maximal $2$-distance sets.  
 
Suppose $\mathbf{x} \in \R^{2n-1}$ can be added to $\hat{H}(n,2)$ while maintaining $2$-distance. 
By a similar argument to that in Section~\ref{sec:2}, $\mathbf{x}$
forms 
\[
\mathbf{x} \in 
\left(
	\left(
\left(\frac{k}{n}^{n-k+1},\left(\frac{k}{n}-1\right)^{k-1}
	\right)^P,
\frac{1}{n}^n\right)^P, \beta
\right)
\]
for some $\beta \in \R$, and $1\leq k \leq n$. 
For $\mathbf{y} \in \hat{H}(n,2)$, we have 
\[
d(\mb{x}, \mb{y})^2
=-1-\frac{1}{n}+k-\frac{k^2}{n}
+ 2 l_1+4 l_2 +\beta^2 \in \{2,4\}, 
\]
where 
$l_i=|\{ q \colon\, y_q=1, x_q=k/n +1-i \}|\in\{0,1\}$. 
 It therefore holds that 
\[ 
\beta=\pm \sqrt{1+\frac{1}{n}-k+\frac{k^2}{n}}. 
\]
We use the notation
\begin{align*}
X_k^{\pm} &=\left(
	\left(
\left(\frac{k}{n}^{n-k+1},\left(\frac{k}{n}-1\right)^{k-1}
	\right)^P,
\frac{1}{n}^n\right)^P, \pm \sqrt{1+\frac{1}{n}-k+\frac{k^2}{n}}
\right),\\
Y_k^{\pm} &=\left(
\left(\frac{k}{n}^{n-k+1},\left(\frac{k}{n}-1\right)^{k-1}
	\right)^P,
\frac{1}{n}^n, \pm \sqrt{1+\frac{1}{n}-k+\frac{k^2}{n}}
\right),\\
Z_k^{\pm} &=\left(
\frac{1}{n}^n,
\left(\frac{k}{n}^{n-k+1},\left(\frac{k}{n}-1\right)^{k-1}
	\right)^P, \pm \sqrt{1+\frac{1}{n}-k+\frac{k^2}{n}}
\right)
\end{align*}
for $k\in \{1,\ldots,n\}$. 
Since the radicand of $\beta$ is not negative,
an element of $X_k^{\pm}$ can be added to 
$\hat{H}(n,2)$ only for $(k,n)$ such that $k\geq 1$ and $n \leq 5$, and for $(k,n)$ such that 
$k\in \{1,n-1,n,n+1\}$ and $n\geq 6$. 
The following is the classification of the maximal sets, which can be added to the corresponding $\hat{H}(n,2)$ and has size at least $2$.  
\begin{itemize}
\item  $n\geq 3$
\begin{enumerate}
\item $X_{n-1}^+$ or $X_{n-1}^-$ [$n(n-1)$ points] 
\end{enumerate}
\item  $n\geq 2$
\begin{enumerate}
\item  $Y_n^+$, $Y_n^-$, $Z_n^+$, or $Z_n^-$ [$n$ points] 
\end{enumerate}
\item $n=2$
\begin{enumerate}
\item $X_1^+ \cup X_1^{-}$ [2 points]
\end{enumerate}
\item $n=4$
\begin{enumerate}
\item $X_1^+\cup X_1^-$ [2 points]
\item $Y_2^+$, $Y_2^-$, $Z_2^+$, or $Z_2^-$  [4 points]
\item $X \subset X_3^+ \cup X_3^{-}$ [12 points] \\ 
$X$ contains only one of $(\mb{x},\sqrt{1/2})$ or $(\mb{y},-\sqrt{1/2})$ such that $x_i=3/4 \Leftrightarrow y_i=-1/4$, and  $x_i=-1/4 \Leftrightarrow y_i=3/4$.
\end{enumerate}
\item $n=5$
\begin{enumerate}
\item $Y_2^+ \cup Z_3^+$, or $Z_2^+\cup Y_3^+$  [15 points] 
\end{enumerate}
\item $n=8$
\begin{enumerate}
\item $X_1^+ \cup X_9^-$ or $X_1^- \cup X_9^+$  [$2$ points]  
\end{enumerate}
\end{itemize}
For $n\geq 3$ (1), $X_{n-1}^+ \cup \hat{H}(n,2)$ is isometric to $\tilde{J}(2n,2)$. For $n\geq 2$ (1), the graph obtained from $Y_n^{+}\cup \hat{H}(n,2)$ is bi-regular. 
For $n=5$ (1), $Y_2^+\cup Z_3^+ \cup \hat{H}(n,2)$ is still in $\mathbb{R}^8$, and it is also maximal in $\mathbb{R}^9$.

\bigskip

\noindent
\textbf{Acknowledgments.} 
The authors thank Norihide Tokushige for giving 
information about Theorem~\ref{lem:sum3} and thank   anonymous referees for 
careful reading and valuable comments. 
Nozaki is supported by JSPS KAKENHI Grant Numbers 25800011, 26400003, and 16K17569.

\end{document}